\documentclass[12pt]{amsart}

\usepackage{amssymb}
\usepackage{amsthm}
\usepackage{amsmath}
\usepackage{amsfonts}
\usepackage{latexsym} 
\usepackage{eucal}
\usepackage{indentfirst} 
\usepackage{mathrsfs}

\newtheorem{theorem}{Theorem}[section]

\newtheorem{cor}[theorem]{Corollary}
\newtheorem{lemma}[theorem]{Lemma}

\begin{document}

\newcommand{\N}{\mathbb{N}}

\def \cd {, \ldots ,}
\def \lf{\| f \|}
\def \bs{\backslash}
\def \ep{\varepsilon}

\def \F {\mathfrak{F}}
\def \fraka {\mathfrak{a}}
\def \U {\mathfrak{U}}
\def \frakp {\mathfrak{p}}
\def \frakm {\mathfrak{m}}
\def \frakc {\mathfrak{c}}

\def \cinf {C^\infty}
\def \cid {C^\infty (\partial D)}
\def \bbr {\mathbb R}
\def \bbrs {\mathbb R^2}
\def \bbc {\mathbb C}
\def \od {\overline D}
\def \ol {\overline L}
\def \oj {\overline J}
\def \ma {\mathfrak{M}_A}  
\def \mb {\mathfrak{M}_B}
\def \muc {\mathfrak{M}_{\mathscr U}}
\def \uc {\mathscr U}
\def \xy {(x,y)}
\def \fg {f \otimes g}
\def \vp {\varphi}
\def \pl {\partial L}
\def \ad {A(D)}
\def \aid {A^\infty (D)}
\def \ot {\otimes}
\def \rbl {R_b (L)}
\def\Int {{\rm Int}}

\title[Weak Sequential Completeness of Uniform Algebras]{Weak Sequential Completeness of\\ Uniform Algebras}
\author{J. F. Feinstein}
\address{School of Mathematical Sciences, University of Nottingham, University Park, Nottingham NG7 2RD, UK }
\email{Joel.Feinstein@nottingham.ac.uk}
\author{Alexander J. Izzo}
\address{Department of Mathematics and Statistics, Bowling Green State University, Bowling Green, OH 43403}
\thanks{The authors' collaboration was partially supported by the Research in Teams program at the Banff International Research Station.  The first author was also partially supported by a Simons collaboration grant and by NSF Grant DMS-1856010.}
\email{aizzo@bgsu.edu}


\begin{abstract}
\vskip 24pt
We give a simple, elementary proof that a uniform algebra is weakly sequentially complete if and only if it is finite-dimensional.
\end{abstract}

\maketitle

\section{The Result}

For
$X$ a compact Hausdorff space, we denote by $C(X)$ the algebra of all continuous complex-valued functions on $X$ with the supremum norm
$ \|f\|_{X} = \sup\{ |f(x)| : x \in X \}$.  A \emph{uniform algebra} on $X$ is a closed subalgebra of $C(X)$ that contains the constant functions and separates
the points of $X$.
On every compact Hausdorff space $X$ there is the trivial example of a uniform algebra, namely $C(X)$ itself.  By the Stone-Weierstrass theorem, $C(X)$ is the only self-adjoint uniform algebra on the space $X$.  However, there are many other (nonself-adjoint) uniform algebras.  A typical example is the disc algebra which consists of the continuous complex-valued functions on the closed unit disc that are holomorphic on the open unit disc.  The uniform algebras form a class of Banach algebras that is important both in the field of Banach algebras and in complex analysis, and uniform algebras also have applications to operator theory.  In this paper we consider certain Banach space properties of uniform algebras, primarily weak sequential completeness and reflexivity.  (These terms are defined in the next section.)

Every weakly sequentially complete uniform algebra is finite-dimensional.  Although this fact is known to a few experts, the result is certainly not well known and seems not to be explicitly stated in the literature.  In this paper we present a simple, elementary proof of the result.  A different proof, using Arens regularity and bounded approximate identities, is given in the forthcoming book of Garth Dales and Ali \" Ulger \cite[Section~3.6]{Dales-Ulger}.  An anonymous referee has pointed out that using results in the literature a stronger statement can be obtained: every infinite-dimensional uniform algebra contains an isometric copy of the Banach space $c$ of all convergent sequences of complex numbers.  We give the referee's argument near the end of the paper.  

\begin{theorem}\label{main}
Every weakly sequentially complete uniform algebra is finite-dimensional.
\end{theorem}

Since every reflexive Banach space is weakly sequentially complete, we have the following as an immediate consequence.

\begin{cor}
For a uniform algebra $A$, the following are equivalent.
\begin{itemize}
\item[(a)] $A$ is weakly sequentially complete.
\item[(b)] $A$ is reflexive.
\item[(c)] $A$ is finite-dimensional.
\end{itemize}
\end{cor}

The fact that every reflexive uniform algebra is finite-dimensional does appear in the literature.  However, the only explicit mention of this fact that we have found in the literature is at the very end of the paper \cite{Ellis} where the result is obtained as a consequence of the general theory developed in that paper concerning a representation due to Asimow \cite{Asimow} of a uniform algebra as a space of affine functions.  A closely related result, which we will discuss at the end of our paper, appears in the paper \cite{BW} of Paul Beneker and Jan Wiegerinck: no separable infinite-dimensional  uniform algebra is a dual space.  The fact that every reflexive uniform algebra is finite-dimensional follows immediately since every infinite-dimensional uniform algebra contains a separable infinite-dimensional uniform algebra.

One can also consider what are sometimes called {\it nonunital uniform algebras\/}.  These algebras are roughly the analogues on noncompact locally compact Hausdorff spaces of the uniform algebras on compact Hausdorff spaces.  (The precise definition is given in the next section.)  Every nonunital uniform algebra is, in fact, a maximal ideal in a uniform algebra, and hence is, in particular, a codimension 1 subspace of a uniform algebra.  Since it is easily proven that the failure of weak sequential completeness is inherited by finite codimensional subspaces, it follows at once that the above results hold also for nonunital uniform algebras.

It should be noted that the above results do \emph{not} extend to general semisimple commutative Banach algebras.  For instance, for $1\leq p\leq\infty$, the Banach space $\ell^p$ of $p$th-power summable sequences of complex numbers is a Banach algebra under coordinatewise multiplication and is of course well known to be reflexive for $1<p<\infty$; for $p=1$ the space is nonreflexive but is weakly sequentially complete \cite[p.~140]{Wojtaszczyk}.  Also for $G$ an infinite locally compact abelian group, the Banach space $L^1(G)$ is a Banach algebra with convolution as multiplication and is nonreflexive but is weakly sequentially complete \cite[p.~140]{Wojtaszczyk}.  
All of these Banach algebras are nonunital, with the exception of the algebras $L^1(G)$ for $G$ a discrete group. However, adjoining an identity in the usual way where necessary, one obtains from them unital Banach algebras with the same properties with regard to reflexivity and weak sequential completeness.

In the next section, which can be skipped by those well versed in basic uniform algebra and Banach space concepts, we recall some definitions.  The proof of Theorem~\ref{main} is then presented in Section~\ref{theproof}.  A proof of the stronger statement that every infinite-dimensional uniform algebra contains an isometric copy of the Banach space $c$ is given in Section~\ref{c}.  In the concluding Section~\ref{dual-section} we discuss the theorem of Beneker and Wiegerinck that no separable infinite-dimensional uniform algebra is a dual space.

\section{Definitions}~\label{prelim}

Recall from the introduction that a uniform algebra on a compact Hausdorff space $X$ is an algebra of continuous complex-valued functions on $X$ that contains the constant functions, separates the points of $X$, and is (uniformly) closed in the algebra $C(X)$ of all continuous complex-valued functions on $X$.
For $Y$ a noncompact, locally compact Hausdorff space, we denote by $C_0(Y)$ the algebra of continuous complex-valued functions on $Y$ that vanish at infinity, equipped  with the supremum norm.
By a \emph{nonunital uniform algebra} $B$ on $Y$ we mean a closed subalgebra of $C_0(Y)$ that strongly separates points in the sense that for every pair of distinct points $x$ and $y$ in $Y$ there is a function $f$ in $B$ such that $f(x)\neq f(y)$ and $f(x)\neq 0$.  If $B$ is a nonunital uniform algebra on $Y$, then the linear span of $B$ and the constant functions on $Y$ forms a unital Banach algebra that can be identified with a uniform algebra $A$ on the one-point compactification of $Y$, and under this identification $B$ is the maximal ideal of $A$ consisting of the functions in $A$ that vanish at infinity.

Let $A$ be a uniform algebra on a compact Hausdorff space $X$.  A closed subset $E$ of $X$ is a \emph{peak set} for $A$ if there is a function $f\in A$ such that $f(x)=1$ for all $x\in E$ and $|f(y)|< 1$ for all $y\in X\setminus E$.  Such a function $f$ is said to \emph{peak on $E$}.  A \emph{generalized peak set} is an intersection of peak sets.  A point $p$ in $X$ is a \emph{peak point} if the singleton set $\{p\}$ is a peak set, and $p$ is a \emph{generalized peak point} if $\{p\}$ is a generalized peak set.
A closed subset $E$ of $X$ is an \emph{interpolation set} for $A$ if $A|E=C(E)$, where $A|E$ denotes the algebra of restrictions of functions in $A$ to $E$.  The set $E$ is a \emph{peak interpolation set} for $A$ if $E$ is both a peak set and an interpolation set for $A$.
For $\Lambda$ a bounded linear functional on $A$, we say that a complex regular Borel measure $\mu$ on $X$ represents $\Lambda$ if $\Lambda(f)=\int f \, d\mu$ for every $f\in A$.

A Banach space $A$ is \emph{reflexive} if the canonical embedding of $A$ into its double dual $A^{**}$ is a bijection.  The Banach space $A$ is \emph{weakly sequentially complete} if every weakly Cauchy sequence in $A$ is weakly convergent in $A$.  More explicitly the condition is this: for each sequence $(x_n)$ in $A$ such that $(\Lambda x_n)$ converges for every $\Lambda$ in the dual space $A^*$, there exists an element $x$ in $A$ such that $\Lambda x_n \rightarrow \Lambda x$ for every $\Lambda$ in $A^*$.

\section{The Proof}\label{theproof}

Our proof of Theorem~\ref{main} hinges on the following lemma.

\begin {lemma} \label{not-open}
Every infinite-dimensional uniform algebra has a peak set that is not open.
\end {lemma}

The proof of the above lemma uses two preliminary lemmas.

\begin {lemma} \label{char}
Let $A$ be a uniform algebra on a compact Hausdorff space $X$, and let $P$ be an open peak set for $A$.  Then the characteristic function of $P$ lies in $A$.
\end {lemma}

\begin{proof}
Choose a function $f$ that peaks on $P$.  Then the sequence $(f^n)$ of powers of $f$ converges uniformly to the characteristic function $\chi_P$ of $P$, and hence $\chi_P$ lies in $A$.
\end{proof}

\begin {lemma} \label{G}
Every infinite compact Hausdorff space $X$ contains a closed $G_\delta$-set that is not open.
\end {lemma}

\begin{proof}
Let $\{x_n\}$ be a countably infinite subset of $X$.  For each $n=1, 2, 3, \ldots$ choose by Urysohn's lemma a continuous function $f_n:X\rightarrow [0,1]$ such that $f_n(x_k)=0$ for $k<n$ and $f_n(x_n)=1$.  Let $F:X\rightarrow [0,1]^\omega$ be given by $F(x)=\bigl(f_n(x)\bigr)_{n=1}^\infty$.  Then $F(x_m)\neq F(x_n)$ for all $m\neq n$.  Thus the collection $\{F^{-1}(t):t\in [0,1]^\omega\}$ is infinite.  Each of the sets $F^{-1}(t)$ is a closed $G_\delta$-set because $F$ is continuous and $[0,1]^\omega$ is metrizable.  Since these sets form an infinite collection of disjoint sets that cover $X$, they cannot all be open, by the compactness of $X$.
\end{proof}

\begin{proof}[Proof of Lemma~\ref{not-open}]
Let $A$ be an infinite-dimensional uniform algebra on a compact Hausdorff space $X$.

In case $A=C(X)$, the result follows immediately from Lemma~\ref{G}, since in that case it follows from Urysohn's lemma 
that the peak sets of $A$ are exactly the closed $G_\delta$-sets in $X$ (see for instance \cite[Section~33, exercise~4]{Munkres}).  

Now consider the case when $A$ is a proper subalgebra of $C(X)$.  In that case, by the Bishop antisymmetric decomposition \cite[Theorem~2.7.5]{Browder} there is a maximal set of antisymmetry $E$ for $A$ that has more than one point.  Since every maximal set of antisymmetry is a generalized peak set, and every generalized peak set contains a generalized peak point (see the proof of \cite[Corollary~2.4.6]{Browder}), $E$ contains a generalized peak point $p$.  Choose a peak set $P$ for $A$ such that $p\in P$ but $P\nsupseteq E$.  The set $P$ is not open in $X$, for if it were then the characteristic function of $P$ would be in $A$ by Lemma~\ref{char}, which would contradict that $E$ is a set of antisymmetry for $A$.
\end{proof}

\begin{proof}  [Proof of Theorem~\ref{main}]
Let $A$ be an infinite-dimensional uniform algebra on a compact Hausdorff space $X$.  By Lemma~\ref{not-open}, there exists a peak set $P$ for $A$ that is not open.  Choose a function $f\in A$ that peaks on $P$.  

For a bounded linear functional $\Lambda$ on $A$, and a complex regular Borel measure $\mu$ on $X$ that represents $\Lambda$, we have by the Lebesgue dominated convergence theorem that
\begin{equation}\label{Cauchy}
\Lambda(f^n) = \int f^n\, d\mu  \rightarrow \mu(P) \quad {\rm as\quad} n\rightarrow \infty.
\end{equation}
Thus the sequence $(f^n)_{n=1}^\infty$ in $A$ is weakly Cauchy.  Furthermore (\ref{Cauchy}) shows that, regarded as a sequence in the double dual $A^{**}$, the sequence $(f^n)_{n=1}^\infty$ is weak*-convergent to a functional $\Phi\in A^{**}$ that satisfies the equation
$\Phi(\Lambda)=\mu(P)$
for every functional $\Lambda\in A^*$ and every regular Borel measure $\mu$ that represents $\Lambda$.

For $x\in X$, denote the point mass at $x$ by $\delta_x$.  Denote the characteristic function of the set $P$ by $\chi_P$.  Then
\begin{equation}
\Phi(\delta_x)=\chi_P(x)
\end{equation}
while for any function $h\in A$ we have
\begin{equation}
\int h \, d\delta_x=h(x).
\end{equation}
Since $P$ is not open in $X$, the characteristic function 
$\chi_P$ is not continuous and hence is not in $A$.  Consequently, equations (2) and (3) show that the functional $\Phi\in A^{**}$ is not induced by an element of $A$.  We conclude that the weakly Cauchy sequence $(f^n)_{n=1}^\infty$ is not weakly-convergent in $A$.
\end{proof}

\section{Every infinite-dimensional uniform algebra contains $c$}\label{c}

In this section we give a proof of the following theorem along the lines suggested by a referee.

\begin{theorem}\label{referee-theorem}
Every infinite-dimensional uniform algebra contains an isometric copy of the Banach space $c$.
\end{theorem}

As mentioned in the introduction, this theorem strengthens Theorem~\ref{main}.  To see this, first note that the Banach space $c_0$ of sequences of complex numbers converging to zero is not weakly sequentially complete because in $c_0$ the sequence $(1,0,0,0,\ldots)$, ($1,1,0,0,\ldots)$, $(1,1,1, 0,\ldots)$ is weakly Cauchy but not weakly convergent.  Since every norm-closed subspace of a weakly sequentially complete Banach space is itself weakly sequentially complete, it follows immediately that a weakly sequentially complete Banach space can not contain a copy of $c_0$, and hence can not contain a copy of $c$.

The proof of Theorem~\ref{referee-theorem} uses the following two results.  The first of these is due to Alain Bernard while the second is due to Aleksander Pe\l czy\' nski.   Proofs of these results can be found in \cite[pp.~217--219, 241--242]{Stout}. 

\begin{theorem}\label{P1}
If $A$ is a uniform algebra on an infinite compact metrizable space, then there exists an infinite peak interpolation set for $A$.
\end{theorem}

\begin{theorem}\label{P2}
Let $A$ be a uniform algebra on a compact Hausdorff space $X$, and let $K$ be a peak interpolation set for $A$.  Then there exists a linear isometry $L: C(K)\rightarrow A$ such that $(Lf)|K=f$ for all $f\in C(K)$.
\end{theorem}

We also use the following result which is surely known but whose proof we include for the reader's convenience.

\begin{lemma}\label{easy}
Let $S$ be an infinite, compact metrizable space.  Then $C(S)$ contains an isometric copy of the Banach space $c$.
\end{lemma}

\begin{proof}
Choose a sequence of distinct points $s_n$ in $S$ converging to a point $s\in S$.  Set $E=\{s_n: n=1, 2, \ldots\} \cup \{s\}$.  Clearly $C(E)$ is isometric to $c$, and by Theorem~\ref{P2}, for example, there is an isometric copy of $C(E)$ in $C(S)$.
\end{proof}

\begin{proof}[Proof of Theorem~\ref{referee-theorem}]
Let $A$ be an infinite-dimensional uniform algebra on a compact Hausdorff space $X$.  By replacing $A$ by a suitable closed subalgebra, we may assume that $A$ is separable and hence that $X$ is metrizable.  
Then by theorem~\ref{P1}, there exists an infinite subset $K$ of $X$ such that $K$ is a peak interpolation set for $A$.  By Theorem~\ref{P2}, there is an isometric copy of $C(K)$ in $A$, and hence by Lemma~\ref{easy}, $A$ contains an isometric copy of $c$.
\end{proof}

\section{No separable infinite-dimensional  uniform algebra is a dual space}\label{dual-section}

By a theorem of Bessaga and Pe\l czy\' nski \cite[Theorem~10, p. 48]{Diestel}, if the dual space of a Banach space contains an isomorphic copy of the Banach space $c_0$, then it contains an isomorphic copy of $\ell^\infty$.  (Two Banach spaces are \emph{isomorphic} if there is a linear homeomorphism between them.  Isomorphic Banach spaces need not be isometric.)  Thus the following result of Beneker and Wiegerinck \cite{BW} follow immediately from Theorem~\ref{referee-theorem}.

\begin{theorem}\label{not-dual}
No separable infinite-dimensional  uniform algebra is a dual space.
\end{theorem}

Note, however, that there are nonseparable uniform algebras that are dual spaces.  For instance, the uniform algebra $C(\beta \N)$ of all continuous complex-valued functions on the Stone-\v Cech compactification of the positive integers $\N$ can be identified with $\ell^\infty$ and thus is isometrically isomorphic to the dual of $\ell^1$.

Beneker and Wiegerinck obtained Theorem~\ref{not-dual} as a corollary of the main theorem of their paper \cite{BW} which concerns strongly exposed points.  A point $f$ in the closed unit ball $B$ of a Banach space $A$ is said to be \emph{strongly exposed} if there exists $\Lambda\in A^*$ with the properties $\Lambda(f)=\|\Lambda\|=1$ and for every sequence $(g_n)_{n=1}^\infty$ in $A$ such that $\lim_{n\rightarrow\infty}\Lambda(g_n)=\lim_{n\rightarrow\infty}\|g_n\|=1$, we have $\lim_{n\rightarrow\infty} g_n=f$ in $A$.  Beneker and Wiegerinck's main result \cite{BW} states that the unit ball of an infinite-dimensional uniform algebra has no strongly exposed points.  As noted by Beneker and Wiegerinck, Theorem~\ref{not-dual} follows immediately since the the unit ball of a separable dual space is the closed convex hull of its strongly exposed points \cite{Phelps}.  A completely elementary proof of a result stronger than the main theorem of \cite{BW} was later proven by Olav Nygaard and Dirk Werner.  Denoting the real part of a complex number $z$ by ${\rm Re}\, z$,   a \emph{slice} of $B$ is a set of the form
$$S(\Lambda, \ep)= \{g\in B: {\rm Re}\, \Lambda(g)\geq \sup{\rm Re}\, \Lambda(B)-\ep\},$$
for $\Lambda\in A^*$ and $\ep>0$.  Nygaard and Werner \cite{NW} showed that every slice of the closed unit ball of an infinte-dimensional uniform algebra has diameter 2.  There are thus several routes to proving Theorem~\ref{not-dual}.

\bigskip
{\bf Acknowledgements\/}:  The question of whether every reflexive uniform algebra is finite-dimensional was posed to us by Matthias Neufang, and this was the inspiration for the work presented here.  We thank Neufang for his question.  We thank the anonymous referee who brought to our attention Theorem~\ref{referee-theorem}.  We also thank Olav Nygaard for bringing the papers \cite{BW} and \cite{NW} to our attention.

\end{document}